\documentclass[10pt,a4paper,hyperref]{elsarticle}

\usepackage{latexsym,amssymb,amsmath,amsfonts,amsthm}
\usepackage[usenames,dvipsnames]{pstricks}
\usepackage{pst-grad, pst-plot}
\usepackage{epsfig}
\usepackage{hyperref}

\hypersetup{
	colorlinks,
	linktoc=all,
	pdfstartview={FitH}
}

\newtheorem{thm}{Theorem}
\newtheorem{cor}{Corollary}

\newtheorem{lemma}{Lemma}

\newtheorem{clm}{Claim}
\newtheorem{obs}{Observation}

\begin{document}

\title{VPG and EPG bend-numbers of Halin Graphs}

\author{Mathew C. Francis\fnref{ack}}
\ead{mathew@isicnennai.res.in}
\address{Computer Science Unit, Indian Statistical Institute, Chennai}
\fntext[ack]{Partially supported by the DST INSPIRE Faculty Award IFA12-ENG-21.}

\author{Abhiruk Lahiri}
\ead{abhiruk.lahiri@csa.iisc.ernet.in}
\address{Department of Computer Science and Automation, Indian Institute of Science, Bangalore}

\begin{abstract}
A piecewise linear simple curve in the plane made up of $k+1$ line segments, each of which is either horizontal or vertical, with consecutive segments being of different orientation is called a \emph{$k$-bend path}. Given a graph $G$, a collection of $k$-bend paths in which each path corresponds to a vertex in $G$ and two paths have a common point if and only if the vertices corresponding to them are adjacent in $G$ is called a \emph{$B_k$-VPG representation} of $G$. Similarly, a collection of $k$-bend paths each of which corresponds to a vertex in $G$ is called an \emph{$B_k$-EPG representation} of $G$ if any two paths have a line segment of non-zero length in common if and only if their corresponding vertices are adjacent in $G$. The \emph{VPG bend-number} $b_v(G)$ of a graph $G$ is the minimum $k$ such that $G$ has a $B_k$-VPG representation. Similarly, the \emph{EPG bend-number} $b_e(G)$ of a graph $G$ is the minimum $k$ such that $G$ has a $B_k$-EPG representation.
\emph{Halin graphs} are the graphs formed by taking a tree with no degree $2$ vertex and then connecting its leaves to form a cycle in such a way that the graph has a planar embedding. We prove that if $G$ is a Halin graph then $b_v(G) \leq 1$ and $b_e(G) \leq 2$. These bounds are tight. In fact, we prove the stronger result that if $G$ is a planar graph formed by connecting the leaves of any tree to form a simple cycle, then it has a VPG-representation using only one type of 1-bend paths and an EPG-representation using only one type of 2-bend paths.
\end{abstract}

\journal{arXiv.org}

\maketitle

\section{Introduction}
The notion of \textit{edge intersection graphs of paths on a grid} (EPG graphs) was first introduced by Golumbic, Lipshteyn and Stern~\cite{Golumbic}. A graph $G$ is said to be an EPG graph if its vertices can be mapped to paths in a grid graph (a graph with vertex set $\{(x,y)\colon x,y\in\mathbb{N}\}$ and edge set $\{(x,y)(x',y')\colon |x-x'|+|y-y'|=1\}$) in such a way that two vertices in $G$ are adjacent if and only if the paths corresponding to them share an edge of the grid graph. If in addition, none of the paths used have more than $k$ bends (90-degree turns), then $G$ is said to be a $B_k$-EPG graph. We could also think of the paths on the grid graph to be just piecewise linear curves on the plane and this gives rise to the following equivalent definition for EPG graphs. A \emph{$k$-bend path} is piecewise linear curve made up of $k+1$ horizontal and vertical line segments in the plane such that any two consecutive line segments of the curve are of different orientation. A graph $G = (V, E)$ is a $B_k$-EPG graph if there exists a $k$-bend path $P_u$ corresponding to each vertex $u\in V(G)$ such that any two paths $P_u$ and $P_v$ have a horizontal or vertical line segment of non-zero length in common if and only if $uv \in E(G)$. The collection of $k$-bend paths $\{P_u\}_{u\in V(G)}$ is said to be a \emph{$B_k$-EPG representation} of $G$. The \emph{EPG bend-number} $b_e(G)$ of a graph $G$ is the minimum $k$ such that $G$ has a $B_k$-EPG representation. Clearly, $B_0$-EPG graphs are the well known class of interval graphs. Several papers have explored the EPG bend-number of different classes of graphs. The EPG bend-number of trees and cycles have been determined in~\cite{Golumbic}. Bounds on the the EPG bend-number of planar graphs, outerplanar graphs and complete bipartite graphs have been determined in~\cite{HeldtP} and~\cite{Heldt}.

Later in 2012, Asinowski et al.~\cite{AsinG, Asin} introduced the study of \textit{vertex intersection graphs of paths on a grid} (VPG graphs). For defining VPG graphs, we will as before talk about $k$-bend paths in the plane, instead of dealing with paths on a grid graph. A collection of $k$-bend paths $\{P_u\}_{u\in V(G)}$ is said to be a $B_k$-VPG representation of a graph $G$, if for $u,v\in V(G)$, we have $uv\in E(G)$ if and only if $P_u$ and $P_v$ have at least one point in common. A graph is said to be simply a VPG graph if it is a $B_k$-VPG graph for some $k$. Several relationships between VPG graphs and other known graph classes have been studied in~\cite{AsinG}. VPG graphs are known to be equivalent to the known class of string graphs, which are the intersection graphs of simple curves in the plane~\cite{Asin}. $B_0$-VPG graphs are a special case of segment graphs (intersection graphs of line segments in the plane), which are known as $2$-DIR graphs. L-VPG graphs are intersection graphs of L-shaped paths, which is clearly a subclass of $B_1$-VPG graphs. The subclass of L-VPG graphs which have L-VPG representations in which no two paths cross (have a common internal point) are called L-contact graphs~\cite{ChapKU}. More results and a literature survey of graph classes that admit contact VPG representations can be found in~\cite{Aerts}.

A \emph{tree-union-cycle graph} is formed by taking a tree and then connecting its leaves to form a cycle in such a way that the graph has a planar embedding. A \emph{Halin graph} is a tree-union-cycle graph with no degree 2 vertices. These graphs were first introduced by Rudolf Halin in his study of minimally $3$-connected graphs \cite{Halin}.

In this paper we show that Halin graphs are $B_1$-VPG graphs as well as $B_2$-EPG graphs. In fact, we show the stronger results that any tree-union-cycle graph has a L-VPG representation and a $B_2$-EPG representation with one type of 2-bend paths. We also demonstrate Halin graphs that are not $B_0$-VPG graphs and Halin graphs that are not $B_1$-EPG graphs. Note that L-contact graphs are subclass of $2$-degenerate graphs \citep{ChapKU}. As Halin graphs have minimum degree three, it is not possible to obtain a L-contact representation for Halin graphs.

$B_1$-VPG representation of Halin graphs and IO-graphs (2-connected planar graphs in which the interior vertices form a (possibly empty) independent set) was studied in~\cite{Derka}. It was shown in that paper that all IO-graphs are in L-VPG and that all Halin graphs have a $B_1$-VPG representation in which every vertex other than one is represented using an L-shaped path and one vertex is represented using a \rotatebox[origin=c]{180}{L}-shaped path. As any IO-graph can be easily seen to be an induced subgraph of some tree union cycle graph, our results show that both these graph classes are subclasses of L-VPG.

\section{Preliminaries}
A simple graph $G = T \cup C$, where $T$ is a tree and $C$ is a simple cycle induced by all the leaves of the tree $T$ such that $G$ is planar, is called a {\it tree-union-cycle graph}. A {\it Halin graph} is a tree-union-cycle graph with no degree $2$ vertex. 

Let us consider a tree-union-cycle graph $G = T \cup C$, where the cycle $C$ is of length $k$. Let $C$ be $c_0c_1\dots c_{k-1}c_0$. Let $S = V(G) \setminus V(C)$ be the set of internal vertices of the tree $T$. Let us also assume the set $S$ is not a singleton set, i.e. the graph $G$ is not a wheel graph. Designate an internal vertex $r$ of $T$, to be the root of the tree $T$. Once a root is fixed, we define the following notations. 
\begin{enumerate}
\item For any two vertices $u, v \in T$, $u$ is said to be an ancestor of $v$ if $u$ lies on the path $rTv$. On the other hand if $v$ lies on the path $rTu$ then $u$ is called a descendant of $v$. If $uv \in T$ and $v$ is a descendant of $u$ then we often call $v$ is a child of $u$ and $u$ is the parent of $v$.
\item For any vertex $u \in T$, let $L^r(u)$ be the set of all leaves that are descendants of $u$. If $u$ is a leaf then define $L^r(u) = \{ u\}$. Clearly if $u$ is a descendant of $v$ then $L^r(u) \subseteq L^r(v)$. 
\item For each vertex $u \in T$, define the height of the vertex $h^r(u)$ as the length of the path $rTu$. Let $hca(u, v) = \min\{ h^r(x) \colon x \in uTv\}$.
\end{enumerate}
\begin{lemma}
\label{lem: ancdec}
Let $u, v \in S$, then $u$ is an ancestor of $v$ or $v$ is an ancestor of $u$ if and only if $L^r(u) \cap L^r(v) \neq \emptyset$.
\end{lemma}
\begin{proof}
Let $u$ is an ancestor of $v$. Clearly $L^r(v) \subseteq L^r(u)$. Hence $L^r(u) \cap L^r(v) \neq \emptyset$. Similarly it can be shown if $v$ is an ancestor of $u$ then $L^r(u) \cap L^r(v) \neq \emptyset$. Conversely, let $L^r(u) \cap L^r(v) \neq \emptyset$. Hence there exists $c_i \in L^r(u) \cap L^r(v)$. This implies that $u$ and $v$ both are on the path $c_iTr$. Hence one of them is an ancestor of the other. This completes the proof.
\end{proof}

\begin{obs}
For any $u\ (\neq r) \in S$, $L^r(u)$ does not contain all the leaves of $T$. 
\end{obs}
\begin{proof}
For the sake of contradiction, assume if $L^r(u) = L^r(r) = V(C)$. Then it implies that the degree of $r$ is one, which contradicts the fact that $r$ is an internal vertex. 
\end{proof}

The following lemma can be easily seen.

\begin{lemma}
\label{lem: cyc}
Fix a vertex $r \in S$, as the root of the tree $T$. Then the vertices in $L^r(u)$ appear consecutively in the cyclic ordering $c_0, c_1, \dots ,c_{k-1}, c_0$.
\end{lemma}

\begin{lemma}
\label{lem: lin}
Let $c_i, c_{i+1}$ be two leaves and $r$ be an internal vertex on the path $c_iTc_{i+1}$. If we fix root as $r$, then for any $u \in S$, the vertices in $L^r(u)$ appear consecutively in the linear ordering $c_{i+1}, c_{i+2}, \dots,$ $ c_{k-1}, c_0, c_1, \dots , c_i$.
\end{lemma}
\begin{proof}
Suppose there exists a vertex $u\in S$ such that the vertices in $L^r(u)$ do not appear consecutively in the linear ordering $c_{i+1}, c_{i+2}, \dots, c_{k-1}, c_0, c_1, \dots , c_i$. Then it follows from Lemma~\ref{lem: cyc} that $c_{i},c_{i+1}\in L^r(u)$. If $u = r$, then all leaves are descendant of $u$. Hence $u \neq r$. If $v$ is a vertex such that $c_{i+1}\in L^r(v)$, then $v$ is on the path $c_{i+1}Tr$. Similarly, if $w$ is a vertex such that $c_i\in L^r(w)$, then $w$ is on the path $c_iTr$. Therefore, $u$ is an internal vertex of both $c_{i+1}Tr$ and $c_iTr$. 
Since $r \in c_iTc_{i+1}$, $c_iTr$ and $c_{i+1}Tr$ have no internal vertex in common. This implies that such a $u$ cannot exist.
\end{proof}

Now we prove the following lemma which is used to define a linear ordering of the vertices on the cycle. 

\begin{lemma}
\label{lem: path}
Fix a vertex $r \in S$ as the root of $T$. Then there exists an index $i$, such that all the internal vertices except one on the path $c_iTc_{i+1}$ are of degree two.
\end{lemma}
\begin{proof}
Let us consider two consecutive leaves $c_i$ and $c_{i+1}$ of $T$ such that $h^r(hca(c_i, c_{i+1}))$ is maximum. We claim that this pair of leaves $c_i$ and $c_{i+1}$ satisfy the statement of the lemma. Let $u = hca(c_i, c_{i+1})$. If $u$ is the common parent of $c_i$ and $c_{i+1}$, then our claim holds trivially. 

For the sake of contradiction, assume $u$ is not the common parent of $c_i$ and $c_{i+1}$ and at least one of the paths from $uTc_i$ and $uTc_{i+1}$ has an internal vertex of degree at least three. Let that vertex be $v$, lying on the path $uTc_i$. Consider $L^r(v)$. Clearly $c_i$ is in $L^r(v)$. We know that the degree of $v$ is at least three. So $v$ has at least one leaf other than $c_i$ as its descendant. From Lemma~\ref{lem: cyc}, we know that the vertices in $L^r(v)$ appear consecutively in the cyclic ordering of leaves. So either of $c_{i-1}$ or $c_{i+1}$ is in $L^r(v)$. We have $c_{i+1} \notin L^r(v)$. Otherwise, $v$ is a common ancestor of both $c_i$ and $c_{i+1}$ with $h^r(v) > h^r(u)$, which contradicts the fact that $u = hca(c_i, c_{i+1})$. So $c_{i+1} \notin L^r(v)$. Therefore, we can conclude that $c_{i-1} \in L^r(v)$. But then we have found a pair of leaves $c_{i-1}$ and $c_i$ such that $h^r(hca(c_{i-1}, c_i)) \geq h^r(v) > h^r(u)$. This contradicts our initial choice of the pair of leaves $c_i, c_{i+1}$. So there is no internal vertex with degree more than two on the path $uTc_i$. With a similar argument we can conclude that there is no internal vertex with degree more than two on the path $uTc_{i+1}$. This concludes the proof.
\end{proof}

\section{Halin graphs are L-VPG graphs}
Let $c_i, c_{i+1}$ be two leaves satisfying Lemma~\ref{lem: path}. Define, $b_j = c_{(i+j)\mod k}$, for $0 \leq j \leq k-1$. Fix $r^\prime = hca(c_i, c_{i+1})$ to be the new root of $T$.

The following observations are easy consequences of our choice of root and Lemma~\ref{lem: lin}.

\begin{obs}
\label{obs:oneleaf}
Every internal vertex of both the paths $b_0Tr^\prime$ and $b_1Tr^\prime$ has exactly one descendant leaf.
\end{obs}

\begin{obs}
\label{obs:lin}
For any $u \in S$, the vertices in $L^{r^\prime}(u)$ appear consecutively in the linear ordering $b_0, b_1, \dots,$  $ b_{k-1}$.
\end{obs}

Let us define an L-shaped curve on the plane as the set of points given by $\mathcal{L}([x_1, x_2], [y_1, y_2]) = \{(x, y) \vert x = x_1\ \mbox{and}\ y \in [y_1, y_2]\ \mbox{or}\ y = y_1\ \mbox{and}\ x \in [x_1, x_2]\}$. By definition, a graph is an L-VPG graph if one can associate an L-shaped curve $\mathcal{L}_u$ to each vertex $u$ of the graph such that $uv$ is an edge of the graph if and only if $\mathcal{L}_u \cap \mathcal{L}_v \neq \emptyset$. Now we show that $G$ is an L-VPG graph by defining the L-shaped curve $\mathcal{L}_u$ to be associated with each vertex $u \in V(G)$. For any vertex $u\in V(G)$ with $\mathcal{L}_u=\mathcal{L}([x_1,x_2],[y_1,y_2])$, we define $l_x(u)=x_1$, $r_x(u)=x_2$, $l_y(u)=y_1$ and $r_y(u)=y_2$. Define $h = \max \{ h^{r^\prime}(u) \vert u \in V(G)\}$. 

Define $\mathcal{L}_{b_0} = \mathcal{L}([1.5, k-1], [0, h-h^{r^\prime}(b_0) + 2])$.

Define $\mathcal{L}_{b_1} = \mathcal{L}([1, 2], [1, h-h^{r^\prime}(b_1) + 2])$.

Define $\mathcal{L}_{b_{k-1}} = \mathcal{L}([k-1, k], [0, h-h^{r^\prime}(b_{k-1}) + 2])$.

For every leaf $b_i$ other than $b_0, b_1$ or $b_{k-1}$, define $\mathcal{L}_{b_i} = \mathcal{L}([i, i+1], [1, h - h^{r^\prime}(b_i)+2])$.

For every internal vertex $v$, define $\mathcal{L}_v=([\min\{l_x(w)~\vert~w \in L^{r^\prime}(v)\},\max\{l_x(w)~\vert~w \in L^{r^\prime}(v)\}],[h-h^{r^\prime}(v)+1, h-h^{r^\prime}(v)+2])$.

Now we prove that the collection of L-shapes $\{\mathcal{L}_u\}_{u\in V(G)}$ forms a valid one bend $VPG$ representation of $G$.

We first state some observations that will be needed.

\begin{obs}
\label{obs: order}
For any $u \in S$ and leaves $b_i,b_j,b_k$, if $l_x(b_i) < l_x(b_j) < l_x(b_k)$ and $b_i, b_k \in L^{r^\prime}(u)$ then $b_j\in L^{r^\prime}(u)$.
\end{obs}
\begin{proof}
Suppose for the sake of contradiction that there exist $u\in S$ and leaves $b_i,b_j,b_k$ with $l_x(b_i)<l_x(b_j)<l_x(b_k)$ such that $b_i,b_k\in L^{r^\prime}(u)$ and $b_j\notin L^{r^\prime}(u)$. Clearly, this cannot be true if $u=r^\prime$ as $b_j\in L^{r^\prime}(r^\prime)$. Therefore, it must be the case that $u\neq r^\prime$.
If $u$ is an ancestor of $b_0$, then from Observation~\ref{obs:oneleaf}, we know that $|L^{r^\prime}(u)|=1$, which contradicts the fact that $b_i,b_k\in L^{r^\prime}(u)$. Therefore, $u$ cannot be an ancestor of $b_0$. For the same reason, $u$ also cannot be an ancestor of $b_1$. This means that $i\notin \{0,1\}$.
From our construction, we have $l_x(b_2) < l_x(b_3) < \dots < l_x(b_{k-1})$. Since $i\geq 2$ and $l_x(b_i)<l_x(b_j)<l_x(b_k)$, we have $i<j<k$. This implies we have three leaves $b_i,b_j,b_k$ with $i<j<k$ such that $b_i,b_k\in L^{r^\prime}(u)$ and $b_j\notin L^{r^\prime}(u)$. This contradicts Observation~\ref{obs:lin}.
\end{proof}

\begin{obs}
\label{obs:comparable}
Let $u,v\in V(G)$ be distinct vertices such that $u$ is an ancestor of $v$. Then,
\begin{enumerate}[(a)]
\item $l_y(v)\notin [l_y(u),r_y(u)]$,
\item $l_x(v)\in [l_x(u),r_x(u)]$, and
\item $l_y(u)\in [l_y(v),r_y(v)]$ if and only if $u$ is the parent of $v$.
\end{enumerate}
\end{obs}

\begin{proof}
According to our construction, for any internal vertex $w$, $l_y(w) = h - h^{r^\prime}(w) + 1$. When $u$ is an ancestor of $v$ clearly $u$ is an internal vertex. Suppose first that $v$ is a leaf. Clearly, $l_y(v) =0\ \mbox{or}\ 1$. But $h^{r^\prime}(u)$ can be at most $h-1$ and therefore $l_y(u) \geq 2$. Hence $l_y(v) \notin [l_y(u),r_y(u)]$.  If $v$ is an internal vertex then $h^{r^\prime}(u) < h^{r^\prime}(v)$. Hence $l_y(v) < l_y(u)$, implies $l_y(v) \notin [l_y(u),r_y(u)]$. This completes the proof of (a). Now, we prove the Observation~\ref{obs:comparable}(b). Since $u$ is an ancestor of $v$, we have $L^{r^\prime}(v) \subseteq L^{r^\prime}(u)$. From our construction, $l_x(u)=\min\{l_x(w)~\vert~w\in L^{r^\prime}(u)\}\leq l_x(v)$ and $r_x(u)=\max\{l_x(w)~\vert~w\in L^{r^\prime}(u)\}\geq l_x(v)$. So we have $l_x(v)\in [l_x(u),r_x(u)]$. This proves (b). In order to prove (c), first assume $u$ to be the parent of $v$. Then $u$ is always an internal vertex. From our construction we have $l_y(u) = r_y(v)$. This implies $l_y(u) \in [l_y(v), r_y(v)]$. Conversely, let $l_y(u) \in [l_y(v), r_y(v)]$. From our construction we have $r_y(v) = h-h^{r^\prime}(v)+2$. Since $u$ is an ancestor of $v$, we have $h^{r^\prime}(u)\leq h^{r^\prime}(v)-1$. If $h^{r^\prime}(u)< h^{r^\prime}(v)-1$, then we have $r_y(v) < l_y(u)$. But this contradicts the fact that $l_y(u) \in [l_y(v), r_y(v)]$. We can thus conclude that $h^{r^\prime}(u)=h^{r^\prime}(v)-1$, implying that $u$ is the parent of $v$. This completes the proof of (c). 
\end{proof}

\begin{obs}
\label{obs:intleaf}
Let $u,v\in V(G)$ such that $u$ is an internal vertex and $v$ is a leaf. Then,
\begin{enumerate}[(a)]
\item $l_y(v)\notin [l_y(u),r_y(u)]$, and
\item $l_x(v)\in [l_x(u),r_x(u)]$ if and only if $u$ is an ancestor of $v$.
\end{enumerate}
\end{obs}
\begin{proof}
According to our construction, $l_y(u) = h - h^{r^\prime}(u) + 1$. For any leaf $v$, $l_y(v) =0\ \mbox{or}\ 1$. Also we know, $h^{r^\prime}(u) \leq h-1$. Hence $l_y(v) < l_y(u)$. This proves (a). In order to prove (b), first assume that $u$ is an ancestor of $v$. From Observation~\ref{obs:comparable}(b), we have $l_x(v) \in [l_x(u), r_x(u)]$. Conversely, assume $l_x(v) \in [l_x(u), r_x(u)]$. Suppose that $u$ is not an ancestor of $v$. Then $v \notin L^{r^\prime}(u)$. (Note that this means that $u\neq r^{\prime}$.) From our construction, $l_x(u) = \min\{l_x(w)~\vert~w \in L^{r^\prime}(u)\}$ and $r_x(u) = \max\{l_x(w)~\vert~w \in L^{r^\prime}(u)\}$. Let $b_i\in L^{r^\prime}(u)$ be such that $l_x(u)=l_x(b_i)$ and let $b_j\in L^{r^\prime}(u)$ be such that $r_x(u)=l_x(b_j)$. So we have found three leaves $b_i, v$ and $b_j$ such that $l_x(b_i) \leq l_x(v) \leq l_x(b_j)$. Note that from our construction, for any two leaves $w,w'$, we have $l_x(w)=l_x(w')$ if and only if $w=w'$. Since $b_i, b_j \in L^{r^\prime}(u)$ and $v \notin L^{r^\prime}(u)$, we can conclude that the vertices $b_i,b_j$ are distinct from $v$. Therefore, we have $l_x(b_i) < l_x(v) < l_x(b_j)$. This contradicts Observation~\ref{obs: order}. Hence $u$ is an ancestor of $v$. This completes the proof.
\end{proof}

\begin{obs}
\label{obs:bothint}
Let $u$ and $v$ be both internal vertices. If $u$ is not an ancestor of $v$, or $v$ is not an ancestor of $u$, then $[l_x(v),r_x(v)]\cap [l_x(u),r_x(u)]=\emptyset$.
\end{obs}
\begin{proof}
Let $u$ and $v$ be internal vertices such that neither of them is an ancestor of the other. Hence from the Lemma~\ref{lem: ancdec}, we have $L^{r^\prime}(u) \cap L^{r^\prime}(v) = \emptyset$. For the sake of contradiction, assume $[l_x(v),r_x(v)]\cap [l_x(u),r_x(u)]\neq\emptyset$. Then either $l_x(u) \in [l_x(v),r_x(v)]$ or $l_x(v) \in [l_x(u),r_x(u)]$. Let us consider the case when $l_x(u) \in [l_x(v),r_x(v)]$. From our construction, $l_x(u) = \min\{l_x(w)~\vert~w \in L^{r^\prime}(u)\}$, $l_x(v)=\min\{l_x(w)~\vert~w \in L^{r^\prime}(v)\}$ and $r_x(v) = \max\{l_x(w)~\vert~w \in L^{r^\prime}(v)\}$. Let $b_i,b_j$ and $b_k$ be leaves such that $l_x(b_i) = l_x(v)$, $l_x(b_j) = l_x(u)$ and $l_x(b_k) = r_x(v)$. Clearly, $b_i, b_k \in L^{r^\prime}(v)$ and $b_j \in L^{r^\prime}(u)$. Since $L^{r^\prime}(u)\cap L^{r^\prime}(v)=\emptyset$, we know that $b_j\notin L^{r^\prime}(v)$. This implies that $b_j$ is distinct from both $b_i$ and $b_k$. Since $l_x(v)\leq l_x(u)\leq r_x(v)$, we have $l_x(b_i) \leq l_x(b_j) \leq l_x(b_k)$. From our construction, it is clear that for any two leaves $w,w'$, we have $l_x(w)=l_x(w')$ if and only if $w=w'$. Therefore, since $b_j$ is distinct from $b_i,b_k$, we have $l_x(b_i) < l_x(b_j) < l_x(b_k)$. Since we also have $b_i, b_k \in L^{r^\prime}(v)$ and $b_j \notin L^{r^\prime}(v)$, we have a contradiction to Observation~\ref{obs: order}. The proof for the case when $l_x(v) \in [l_x(u),r_x(u)]$ is similar.
\end{proof}

\begin{obs}
\label{obs:bothleaf}
Let $u$ and $v$ be both leaves. Then $u$ and $v$ are consecutive leaves if and only if either: 
\begin{enumerate}[(a)]
\item $l_x(u) \in [l_x(v), r_x(v)]$ and $l_y(v) \in [l_y(u), r_y(u)]$, or
\item $l_x(v) \in [l_x(u), r_x(u)]$ and $l_y(u) \in [l_y(v), r_y(v)]$.
\end{enumerate}
\end{obs}
\begin{proof}
Suppose that $u=b_i$ and $v=b_j$ are consecutive leaves. Without loss of generality we assume that $i<j$. Clearly, $i<k-1$. First, consider the case when $i>0$. Then, we have $j>1$. From the construction, it is clear that $l_x(v)=r_x(u)$ and $l_y(u) \in [l_y(v), r_y(v)]$. Hence it satisfies (b). If $i=0$ (that is $u=b_0$), either $v=b_1$ or $v=b_{k-1}$. If $u = b_0$ and $v = b_1$, then $l_x(u) \in  [l_x(v), r_x(v)]$ and $l_y(v) \in [l_y(u), r_y(u)]$. Hence it satisfies (a). If $u = b_0$ and $v=b_{k-1}$, then $l_x(v) = r_x(u)$ and $l_y(u) = l_y(v)$. In this case, (b) is satisfied. So at least one of (a) or (b) is true when $u$ and $v$ are consecutive leaves. Conversely, let us assume $u = b_i$ and $v = b_j$ are not consecutive leaves. Again, we assume without loss of generality that $i<j$. First consider the case when $i>0$. From the construction, it is clear that if $j >i+1$, then $[l_x(u), r_x(u)] \cap [l_x(v), r_x(v)] = \emptyset$. This implies that neither of (a) or (b) is satisfied. Now consider the case when $u = b_0$ and $v \neq b_1~\mbox{or}~k-1$. In this case, $l_y(u) \notin [l_y(v), r_y(v)]$ and $l_x(u) \notin [l_x(v), r_x(v)]$. Hence neither of (a) or (b) is satisfied. This concludes the proof.
\end{proof}

Now we are ready to show the main result of this section.

\begin{clm}
\label{clm:halin1vpg}
The collection of L-shapes $\{\mathcal{L}_u\}_{u\in V(G)}$ forms a valid one bend VPG representation of $G$.
\end{clm}
\begin{proof}
We prove that for distinct $u, v \in V(G)$, $uv \in E(G)$ if and only if $\mathcal{L}_u \cap \mathcal{L}_v \neq \emptyset$.
It is easy to see that $\mathcal{L}_u\cap\mathcal{L}_v\neq\emptyset$ if and only if $\mathcal{L}_u,\mathcal{L}_v$ satisfy at least one of the following two conditions:
\begin{enumerate}
\renewcommand{\theenumi}{(\arabic{enumi})}
\renewcommand{\labelenumi}{(\arabic{enumi})}
\item\label{cons:uv} $l_x(u) \in [l_x(v), r_x(v)]$ and $l_y(v) \in [l_y(u), r_y(u)]$, or
\item\label{cons:vu} $l_x(v) \in [l_x(u), r_x(u)]$ and $l_y(u) \in [l_y(v), r_y(v)]$.
\end{enumerate}

Therefore, we only need to show that for distinct $u,v\in V(G)$, $uv\in E(G)$ if and only if $\mathcal{L}_u,\mathcal{L}_v$ satisfy either condition~\ref{cons:uv} or condition~\ref{cons:vu}. Note that we will be done if we show that $\mathcal{L}_u,\mathcal{L}_v$ satisfy either condition~\ref{cons:uv} or condition~\ref{cons:vu} if and only if $u$ is a parent of $v$, $v$ is a parent of $u$, or $u, v$ are consecutive leaves. 

Let $u$ be the parent of $v$. From Observation~\ref{obs:comparable}(c), we have $l_y(u) \in [l_y(v), r_y(v)]$. Also, from Observation~\ref{obs:comparable}(b), we have $l_x(v) \in [l_x(u), r_x(u)]$. Therefore, $\mathcal{L}_u,\mathcal{L}_v$ satisfy condition~\ref{cons:vu}. Similarly, it can be shown that when $v$ is the parent of $u$, $\mathcal{L}_u,\mathcal{L}_v$ satisfy the condition~\ref{cons:uv}. When $u, v$ are consecutive leaves, from Observation~\ref{obs:bothleaf}, we have that $\mathcal{L}_u,\mathcal{L}_v$ satisfy either condition~\ref{cons:uv} or condition~\ref{cons:vu}.

Conversely, we show that if $u$ is not a parent of $v$, $v$ is not a parent of $u$ and $u, v$ are not two consecutive leaves, then $\mathcal{L}_u,\mathcal{L}_v$ do not satisfy either condition~\ref{cons:uv} or condition~\ref{cons:vu}. Suppose first that neither of $u, v$ is an ancestor of the other. If $u$ is an internal vertex and $v$ is a leaf, then by Observation~\ref{obs:intleaf}(a), we have $l_y(v)\notin [l_y(u),l_y(u)]$, implying that $\mathcal{L}_u,\mathcal{L}_v$ do not satisfy condition~\ref{cons:uv}. By Observation~\ref{obs:intleaf}(b), we have $l_x(v)\notin [l_x(u),r_x(u)]$, implying that $\mathcal{L}_u,\mathcal{L}_v$ do not satisfy condition~\ref{cons:vu} either. Similarly, it can be shown that when $v$ is an internal vertex and $u$ is a leaf, $\mathcal{L}_u,\mathcal{L}_v$ do not satisfy either condition~\ref{cons:uv} or condition~\ref{cons:vu}.
Let $u$ and $v$ both be internal vertices. Then, by Observation~\ref{obs:bothint}, we have $l_x(u) \notin [l_x(v), r_x(v)]$ and $l_x(v) \notin [l_x(u), r_x(u)]$. Therefore, $\mathcal{L}_u,\mathcal{L}_v$ do not satisfy either condition~\ref{cons:uv} or condition~\ref{cons:vu}. Let us now look at the case when $u$ and $v$ are both leaves. Then, Observation~\ref{obs:bothleaf} implies that $\mathcal{L}_u,\mathcal{L}_v$ do not satisfy either condition~\ref{cons:uv} or condition~\ref{cons:vu}.

Suppose that $u$ is an ancestor of $v$. Then, by Observation~\ref{obs:comparable}(a), we know that $\mathcal{L}_u,\mathcal{L}_v$ do not satisfy condition~\ref{cons:uv}. Also, since $u$ is not the parent of $v$, we know by Observation~\ref{obs:comparable}(c) that $\mathcal{L}_u,\mathcal{L}_v$ do not satisfy condition~\ref{cons:vu} either. If $v$ is an ancestor of $u$, then the same line of reasoning shows that $\mathcal{L}_u,\mathcal{L}_v$ do not satisfy either condition~\ref{cons:uv} or condition~\ref{cons:vu}.

This completes the proof of Claim~\ref{clm:halin1vpg}.
\end{proof}

We now show that there are Halin graphs that have VPG bend-number more than 0.

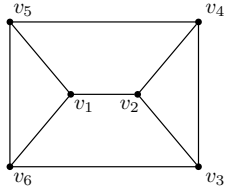
\begin{figure}[h]
\centering
\scalebox{0.8}
{
\begin{pspicture}(0,-1.558125)(4.0028124,1.558125)
\psline[linewidth=0.02cm,dotsize=0.07055555cm 2.0]{*-}(2.1809375,-0.0403125)(1.0809375,-0.0403125)
\psline[linewidth=0.02cm,dotsize=0.07055555cm 2.0]{*-}(1.0809375,-0.0403125)(0.0809375,-1.2403125)
\psline[linewidth=0.02cm,dotsize=0.07055555cm 2.0]{*-}(0.0809375,1.1596875)(1.0809375,-0.0403125)
\psline[linewidth=0.02cm,dotsize=0.07055555cm 2.0]{*-}(3.1809375,-1.2403125)(2.1809375,-0.0403125)
\psline[linewidth=0.02cm,dotsize=0.07055555cm 2.0]{*-}(3.1809375,1.1596875)(2.1809375,-0.0403125)
\psline[linewidth=0.02cm,dotsize=0.07055555cm 2.0]{*-}(0.0809375,-1.2403125)(3.1809375,-1.2403125)
\psline[linewidth=0.02cm](3.1809375,1.1596875)(3.1809375,-1.2403125)
\psline[linewidth=0.02cm](3.1809375,1.1596875)(0.0809375,1.1596875)
\psline[linewidth=0.02cm](0.0809375,1.1596875)(0.0809375,-1.2403125)
\rput(1.3023437,-0.2353125){$v_1$}
\rput(2.0523438,-0.2353125){$v_2$}
\rput(3.4623436,-1.4353126){$v_3$}
\rput(3.4623436,1.3546875){$v_4$}
\rput(0.30234374,1.3546875){$v_5$}
\rput(0.30234374,-1.4353126){$v_6$}
\end{pspicture} 
}
\caption{A Halin graph which has no $B_0$-VPG representation.}
\label{fig: fig01}
\end{figure}

\begin{clm}\label{clm:halinnotb0vpg}
There exists a Halin graph that is not $B_0$-VPG.
\end{clm}

\begin{proof}
Consider the Halin graph with six vertices and nine edges, shown in Figure~\ref{fig: fig01}. For the sake of contradiction, assume that this graph is in $B_0$-VPG. Then it has a VPG representation in which each vertex $v$ is represented by a horizontal or vertical line segment, which we denote by $\mathcal{S}_v$. If $(x_1,y_1)$ and $(x_2,y_2)$ are the two endpoints of a segment $\mathcal{S}_v$ in this representation, we define $l_x(v)=\min\{x_1,x_2\}$, $r_x(v)=\max\{x_1,x_2\}$, $l_y(v)=\min\{y_1,y_2\}$ and $r_y(v)=\{y_1,y_2\}$. Clearly, if $\mathcal{S}_v$ is a horizontal segment, we have $l_y(v)=r_y(v)$ and if it is a vertical segment, we have $l_x(v)=r_x(v)$.

Notice that $v_1, v_5, v_6$ induce a $K_3$ in the graph. Clearly, some two among $\mathcal{S}_{v_5}$, $\mathcal{S}_{v_6}$ and $\mathcal{S}_{v_1}$ have the same orientation, i.e., either horizontal or vertical. Because the vertices are all symmetric to each other, we can assume without loss of generality that $\mathcal{S}_{v_5}$ and $\mathcal{S}_{v_6}$ are both horizontal. (We can argue in a similar fashion if they are both vertical segments.) As they intersect, we have $l_y(v_5) = l_y(v_6)$ and $[l_x(v_5), r_x(v_5)] \cap [l_x(v_6), r_x(v_6)] \neq \emptyset$. Let $[l_x(v_5), r_x(v_5)] \cap [l_x(v_6), r_x(v_6)] = [a, b]$. Again by symmetry between $v_5$ and $v_6$, we can assume without loss of generality that $l_x(v_5)\leq l_x(v_6)$. 
% .
Note that $v_5$ has one neighbour $v_4$, which has no adjacency with $v_6$ and $v_6$ has one neighbour $v_3$, which has no adjacency with $v_5$. This implies that $l_x(v_6) = a$ and $r_x(v_5) = b$ and also that there exists a point $(e,l_y(v_5))\in\mathcal{S}_{v_4}$ with $l_x(v_5)\leq e<a$. If $\mathcal{S}_{v_4}$ is a vertical line segment, then this means that $r_x(v_4)=l_x(v_4)=e<a$. Suppose that $\mathcal{S}_{v_4}$ is a horizontal line segment. Then $l_y(v_4)=r_y(v_4)=l_y(v_5)$. Since $\mathcal{S}_{v_4}$ contains the point $(e,l_y(v_5))$, it cannot be the case that $r_x(v_4)\geq a$ as that would imply that the point $(a,l_y(v_5))\in\mathcal{S}_{v_4}$ which would be a contradiction to the fact that $\mathcal{S}_{v_4}\cap\mathcal{S}_{v_6}=\emptyset$. Therefore, we can conclude that $r_x(v_4)<a$. Therefore, irrespective of whether $\mathcal{S}_{v_4}$ is horizontal or vertical, we have $r_x(v_4)<a$. Arguing very similarly, we can also conclude that $l_x(v_3)>b$. So, we have $r_x(v_4)<a\leq b<l_x(v_3)$ and therefore, $\mathcal{S}_{v_4}\cap\mathcal{S}_{v_3}=\emptyset$. But this contradicts the fact that $v_4$ and $v_3$ are adjacent in the graph.
\end{proof}

\begin{thm}
Tree-union-cycle graphs are in $B_1$-VPG but there is even a Halin graph that is not in $B_0$-VPG.
\end{thm}
\begin{proof}
The proof follows directly from Claims~\ref{clm:halin1vpg} and~\ref{clm:halinnotb0vpg}.
\end{proof}

\begin{cor}
If $H$ is any Halin graph, then $b_v(H)\leq 1$. This bound is tight.
\end{cor}

\section{Halin graphs are \texorpdfstring{$B_2$}{B2}-EPG graphs}
\label{sec: b2}
We relabel the vertices in $C$ as $a_0,a_1,\dots,a_{k-1}$ and reassign the root of the tree $T$ at a new internal vertex. Recall that every internal vertex in the path $b_0Tb_1$ except one has degree two. Let $b'_0$ be the internal vertex of $T$ adjacent to $b_0$ and let $b'_1$ be the internal vertex of $T$ that is adjacent to $b_1$ (note that it is possible to have $b'_0=b'_1$). If $b'_0$ has degree two, then define $a_j=b_{(j+1)\mod k}$, for $0\leq j\leq k-1$, and define $r=b'_1$. If $b'_0$ has degree more than two, then we define $a_j=b_{(k-j)\mod k}$, for $0\leq j\leq k-1$, and define $r=b'_0$. For the rest of the section, wherever we use $r$ it will mean our new choice of the root. The ancestor-descendant relationship and height is also defined with respect to this new root. 

The following is an easy consequence of our choice of root and Lemma~\ref{lem: lin}.

\begin{obs}
\label{obs:newlin}
For any vertex $u\in S$, the vertices in $L^r(u)$ appear consecutively in $a_0, a_1,\dots, a_{k-1}$.
\end{obs}

\begin{obs}
\label{obs:parentroot}
$r$ is the parent of $a_0$.
\end{obs}

Denote by $a'$ the parent of $a_{k-1}$.

\begin{obs}
\label{obs:lastaparent}
Either $a'$ has degree two or $a'=r$. 
\end{obs}
\begin{proof}
If $r=b'_1$, from our labelling, we have $a_{k-1} = b_0$. Clearly $b'_1$ has been chosen as a root because degree of $b'_0$ is two. Hence in this case $a'$ has degree two. If $r=b'_0$, from our labelling, we have $a_{k-1} = b_1$. In this case we know that $b'_0$ has degree more than two. Since exactly one vertex of the path $b'_0Tb'_1$ has degree other than two in $T$, we must either have $b'_0\neq b'_1$, in which case $b'_1=a'$ has degree $2$, or $b'_0=b'_1$, in which case $a'=r$. This completes the proof.
\end{proof}

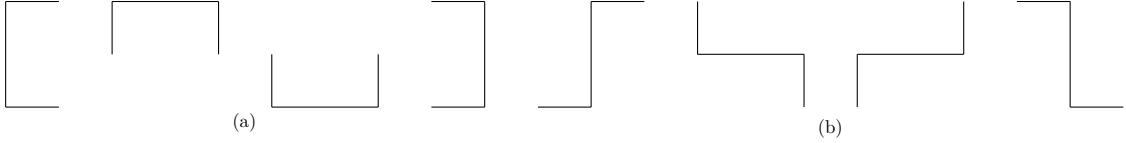
\begin{figure}[h]
\centering
\scalebox{0.7}
{
\begin{pspicture}(0,-1.3284374)(21.01,1.2984375)
\psline[linewidth=0.02cm](0.0,1.2884375)(1.0,1.2884375)
\psline[linewidth=0.02cm](0.0,1.2884375)(0.0,-0.7115625)
\psline[linewidth=0.02cm](0.0,-0.7115625)(1.0,-0.7115625)
\psline[linewidth=0.02cm](8.0,1.2884375)(9.0,1.2884375)
\psline[linewidth=0.02cm](9.0,1.2884375)(9.0,-0.7115625)
\psline[linewidth=0.02cm](9.0,-0.7115625)(8.0,-0.7115625)
\psline[linewidth=0.02cm](2.0,0.2884375)(2.0,1.2884375)
\psline[linewidth=0.02cm](2.0,1.2884375)(4.0,1.2884375)
\psline[linewidth=0.02cm](4.0,1.2884375)(4.0,0.2884375)
\psline[linewidth=0.02cm](5.0,0.2884375)(5.0,-0.7115625)
\psline[linewidth=0.02cm](5.0,-0.7115625)(7.0,-0.7115625)
\psline[linewidth=0.02cm](7.0,0.2884375)(7.0,-0.7115625)
\psline[linewidth=0.02cm](11.0,1.2884375)(12.0,1.2884375)
\psline[linewidth=0.02cm](11.0,1.2884375)(11.0,-0.7115625)
\psline[linewidth=0.02cm](10.0,-0.7115625)(11.0,-0.7115625)
\psline[linewidth=0.02cm](13.0,1.2884375)(13.0,0.2884375)
\psline[linewidth=0.02cm](13.0,0.2884375)(15.0,0.2884375)
\psline[linewidth=0.02cm](15.0,0.2884375)(15.0,-0.7115625)
\psline[linewidth=0.02cm](16.0,-0.7115625)(16.0,0.2884375)
\psline[linewidth=0.02cm](16.0,0.2884375)(18.0,0.2884375)
\psline[linewidth=0.02cm](18.0,1.2884375)(18.0,0.2884375)
\psline[linewidth=0.02cm](19.0,1.2884375)(20.0,1.2884375)
\psline[linewidth=0.02cm](20.0,1.2884375)(20.0,-0.7115625)
\psline[linewidth=0.02cm](20.0,-0.7115625)(21.0,-0.7115625)
\rput(4.485156,-1.0065625){(a)}
\rput(15.495156,-1.1065625){(b)}
\end{pspicture} 
}
\caption{(a) Four different types of C-shaped curves. (b) Four different types of S-shaped curves}
\label{fig: fig03}
\end{figure}

In a $B_2$-EPG representation of a graph, each vertex is represented by a piecewise linear curve in the plane made up of three line segments each of which is horizontal or vertical with consecutive segments being of different orientation. It is not hard to see that such a curve can have any of eight different shapes---i.e., it can be a C-shaped curve and its rotations, or an S-shaped curve and its rotations and reflections (see Figure~\ref{fig: fig03}). We show that $G$ has an EPG representation using C-shaped curves of one type. Later, we show that $G$ also has an EPG representation using S-shaped curves.

A C-shaped curve is defined as a set of points given by $\mathcal{C}(x_1, x_2, x_3, y_1, y_2) = \{(x, y)~\vert~x = x_1\ \mbox{and}\ y \in [y_1, y_2]\} \cup \{ (x, y)~\vert~ y = y_1\ \mbox{and}\ x \in [x_1, x_2]\} \cup \{(x, y)~\vert~ y = y_2\ \mbox{and}\ x \in [x_1, x_3]\}$. Our aim is to associate a C-shaped curve $\mathcal{C}_u$ to each vertex $u\in V(G)$ such that $uv\in E(G)$ if and only if $\mathcal{C}_u\cap \mathcal{C}_v$ contains at least a horizontal or a vertical line segment of non-zero length. For any vertex $u \in V(G)$ with $\mathcal{C}_u = \mathcal{C}(x_1,x_2,x_3,y_1,y_2)$, define $l_x(u)=x_1, p_x(u)=x_2, q_x(u) = x_3, l_y(u)=y_1$ and $r_y(u)=y_2$. Define $h = \max \{ h^{r}(u)~|~u \in V(G)\}$. Let $\epsilon = 1/4$.
\medskip

For every leaf $a_i$ other than $a_0$ and $a_{k-1}$, define $\mathcal{C}_{a_i} = \mathcal{C}(i, i+1+\epsilon, i+3\epsilon, 0, h - h^{r}(a_i)+1)$.

Define $\mathcal{C}_{a_0} = \mathcal{C}(0, 1+\epsilon, k-1+\epsilon, 0, h+1)$.

Define $\mathcal{C}_{a_{k-1}} = \mathcal{C}(k-1, k+\epsilon, k-1+\epsilon, 0, h+1)$.

For every internal vertex $v$ other than $r$ and $a'$, define $\mathcal{C}_v = \mathcal{C}(\min\{i+2\epsilon ~\vert~ a_i \in L^{r}(v)\}, \max\{i+3\epsilon~\vert~ a_i \in L^{r}(v)\},\min\{i+3\epsilon~\vert~ a_i \in L^{r}(v)\}, h-h^{r}(v), h-h^{r}(v)+1)$. 

If $a'\neq r$, define $\mathcal{C}_{a'} = \mathcal{C}(k-1+2\epsilon, k-1+3\epsilon, k-1+3\epsilon, 0, h - h^{r}(a')+1)$ and define $\mathcal{C}_r=\mathcal{C}(2\epsilon,k-1+3\epsilon,1,h,h+1)$.

If $a'=r$, define $\mathcal{C}_r=\mathcal{C}(2\epsilon,k-1+3\epsilon,k-1+\epsilon,h,h+1)$. (Note that we can get a valid representation of $G$ in which no two curves cross more than once by setting $p_x(a')=k-2+3\epsilon$. But here, we are using $p_x(a')=k-1+3\epsilon$ so that the proof becomes a little bit simpler.)

Our aim is now to prove that $\{\mathcal{C}_u\}_{u\in V(G)}$ forms a valid $B_2$-EPG representation of $G$.
We first make some observations.

\begin{obs}
\label{obs: comp}
Let $u$ be the parent of $v$.
\begin{enumerate}[(a)]
\item If $v \notin \{a_0, a_{k-1}\}$, then $l_y(u) = r_y(v)$ and $|[l_x(u), p_x(u)] \cap [l_x(v), q_x(v)]| > 1$, 
\item If $v \in \{a_0, a_{k-1}\}$ and $u= r$, then $r_y(u) = r_y(v)$ and $\vert[l_x(u), q_x(u)] \cap [l_x(v), q_x(v)]| > 1$,  and 
\item If $v = a_{k-1}$ and $u \neq r$, then $l_y(u) = l_y(v)$ and $\vert[l_x(u), p_x(u)] \cap [l_x(v), p_x(v)]| > 1$.
\end{enumerate}
\end{obs}
\begin{proof}
If $v \notin \{a_0, a_{k-1}\}$, we cannot have $u = a'$ and $a' \neq r$. Therefore, from the construction, we have $l_y(u) = h - h^r(u)$ and $r_y(v) = h - h^r(v) +1$. Clearly, since $u$ is the parent of $v$ we have $h^r(v) = h^r(u)+1$. This implies $l_y(u) = r_y(v)$. First, let us consider the case when $v$ is an internal vertex. Since $u$ is the parent of $v$, $L^r(v) \subseteq L^r(u)$. From the construction, we have $l_x(u) = \min\{i+2\epsilon ~\vert~ a_i \in L^{r}(u)\} \leq \min\{i+2\epsilon ~\vert~ a_i \in L^{r}(v)\} = l_x(v)$ and $q_x(v) = \min\{i+3\epsilon~\vert~ a_i \in L^{r}(v)\} \leq \max\{i+3\epsilon~\vert~ a_i \in L^{r}(u)\} = p_x(u)$. (Note that this holds also when $v = a^\prime$, since in that case $a' \neq r$, and therefore $L^r(a^\prime) = \{a_{k-1}\}$. Also note that it is possible that $u = r$.) Clearly from our construction we have $l_x(u) < p_x(u)$ and $l_x(v) < q_x(v)$. Since we have $l_x(u) \leq l_x(v) < q_x(v) \leq p_x(u)$, we conclude that $|[l_x(u), p_x(u)] \cap [l_x(v), q_x(v)]| > 1$. Let us consider the case when $v=a_j$ is a leaf. Clearly, we have $v\in L^r(u)$. Then from the construction, we have $l_x(u) = \min\{i+2\epsilon ~\vert~ a_i \in L^{r}(u)\}$, $p_x(u)=\max\{i+3\epsilon~|~a_i\in L^r(u)\}$, $l_x(v) = j$ and $q_x(v) = j+3\epsilon$. Since $v \in L^r(u)$, it implies that $l_x(u) \leq j+2\epsilon$ and $p_x(u) \geq j+3\epsilon$. Hence $[l_x(u), p_x(u)] \cap [l_x(v), q_x(v)]$ contains the interval $[j+2\epsilon, j+3\epsilon]$. This implies that $|[l_x(u), p_x(u)] \cap [l_x(v), q_x(v)]| > 1$. This completes the proof of (a).

If $u = r$ and $v = a_0$, from the construction we have $r_y(u) = h+1 = r_y(v)$. Again from the construction we have $[l_x(u), q_x(u)] = [2\epsilon, k-1+\epsilon]$ if $u = r = a^\prime$ and $[l_x(u), q_x(u)] = [2\epsilon, 1]$ if $u = r \neq a^\prime$. Also we have $[l_x(v), q_x(v)] = [0, k-1+\epsilon]$. Hence $[2\epsilon, 1] \subset [l_x(u), q_x(u)] \cap [l_x(v), q_x(v)]$, which implies $\vert[l_x(u), q_x(u)] \cap [l_x(v), q_x(v)]\vert > 1$. Now consider the case $u = r$ and $v = a_{k-1}$. As $u$ is the parent of $v$, we have $u = r = a^\prime$. From the construction we have $r_y(u) = h+1 = r_y(v)$, $[l_x(u), q_x(u)] = [2\epsilon, k-1+\epsilon]$ and $[l_x(v), q_x(v)] = [k-1, k-1+\epsilon]$. As $[l_x(u),q_x(u)] \cap [l_x(v),q_x(v)] = [k-1,k-1+\epsilon]$, we have $|[l_x(u),q_x(u)] \cap [l_x(v),q_x(v)]| > 1$. This completes the proof of (b).

If $v= a_{k-1}$ and $u\neq r$, then from the construction we have $l_y(v) = 0$ and $l_y(u) = 0$. From the construction $[l_x(u), p_x(u)] = [k-1+2\epsilon, k-1+3\epsilon]$ and $[l_x(v), p_x(v)] =[k-1, k+\epsilon]$. This implies that $[l_x(u), p_x(u)] \cap [l_x(v), p_x(v)]=[k-1 +2\epsilon , k-1 +3\epsilon]$. Hence $\vert[l_x(u), p_x(u)] \cap [l_x(v), p_x(v)]| > 1$. This completes proof of (c).
\end{proof}

\begin{obs}
\label{obs: inc}
Let $u$ and $v$ be two vertices such that neither of them is an ancestor of the other. Furthermore, let $u$ be an internal vertex. Then,
\begin{enumerate}[(a)]
\item $[l_x(u),p_x(u)] \cap [l_x(v),p_x(v)] = \emptyset$, and
\item If $v \neq a_0$, then $[l_x(u),\max\{p_x(u),q_x(u)\}] \cap [l_x(v),\max\{p_x(v),q_x(v)\}] = \emptyset$.
\end{enumerate}
\end{obs}
\begin{proof}
Let us first prove (a). From Lemma~\ref{lem: ancdec}, we have $L^{r}(u) \cap L^{r}(v) = \emptyset$. First we consider the case when both $u$ and $v$ are internal vertices. Let us define $l_{\min}(u)=\min\{i~|~a_i \in L^r(u)\}$ and $l_{\max}(u)=\max \{i~|~a_i \in L^r(u)\}$. From Observation~\ref{obs:newlin} and $L^{r}(u) \cap L^{r}(v) = \emptyset$, we conclude $[l_{\min}(u),l_{\max}(u)] \cap [l_{\min}(v),l_{\max}(v)] = \emptyset$. Again from our construction, we have $[l_x(w), p_x(w)] = [l_{\min}(w) + 2\epsilon, l_{\max}(w) + 3\epsilon]$, when $w \in S$. Hence we conclude $[l_x(u),p_x(u)] \cap [l_x(v),p_x(v)] = \emptyset$ (note that $\epsilon = 1/4$ implies $3\epsilon < 1$). Let us consider the case when $v = a_j$ is a leaf. As $v$ is not a descendant of $u$, we have $v \notin L^r(u)$. Following Observation~\ref{obs:newlin} this implies that $j \notin [l_{\min}(u),l_{\max}(u)]$. From the construction we have $[l_x(v),p_x(v)] = [j, j+1+\epsilon]$. Again, from the construction we have $[l_x(u),p_x(u)] = [l_{\min}(u)+2\epsilon, l_{\max}(u)+3\epsilon]$. Hence $[l_x(u),p_x(u)] \cap [l_x(v),p_x(v)] = \emptyset$ as $\epsilon = 1/4$. This completes the proof of (a).

Let us now prove (b). Clearly, $u \neq r$ and $v \neq r$, because neither of $u$ or $v$ is a descendant of the other. This tells us that $u,v \notin \{a_0,r\}$. From the construction, we know that for every vertex $w \notin \{a_0,r\}$, $\max\{p_x(w), q_x(w)\} = p_x(w)$. So it is sufficient to show that, $[l_x(u),p_x(u)] \cap [l_x(v),p_x(v)] = \emptyset$. This follows from (a). 
\end{proof}

\begin{obs}
\label{obs: leaves}
Let $u$ and $v$ be both leaves. If $u$ and $v$ are consecutive leaves then either: 
\begin{enumerate}[(a)]
\item $l_y(u) = l_y(v)$ and $\vert[l_x(u), p_x(u)] \cap [l_x(v), p_x(v)]\vert > 1$, or
\item $r_y(u) = r_y(v)$ and $\vert[l_x(u), q_x(u)] \cap [l_x(v), q_x(v)]\vert > 1$.
\end{enumerate}
\end{obs}
\begin{proof}
Suppose $u = a_i$ and $v = a_{i+1}$ are consecutive leaves where $0 \leq i \leq k-2$. From the construction it is clear that $l_y(u) = l_y(v) = 0$ for all $i$. Again from the construction we have $[l_x(u), p_x(u)] \cap [l_x(v), p_x(v)] = [i+1, i+1+\epsilon]$. Hence they satisfy (a). When $u = a_0$ and $v = a_{k-1}$, from the construction it is clear that $r_y(u) = r_y(v) = h+1$. Then $[l_x(u), q_x(u)] \cap [l_x(v), q_x(v)] = [k-1, k-1+\epsilon]$. Hence they satisfy (b).
\end{proof}

We are now ready to prove that $\{\mathcal{C}_u\}_{u\in V(G)}$ forms a valid $B_2$-EPG representation of $G$.

\begin{clm}
\label{clm: halin2epg}
For distinct $u,v\in V(G)$, $uv\in E(G)$ if and only if $\mathcal{C}_u\cap\mathcal{C}_v$ contains a horizontal or a vertical line segment of non-zero length.
\end{clm}
\begin{proof}
It is easy to see that $\mathcal{C}_u \cap \mathcal{C}_v$ contains a horizontal or a vertical line segment of non-zero length if and only if $\mathcal{C}_u, \mathcal{C}_v$ satisfies at least one of the following.
\begin{enumerate}
\renewcommand{\theenumi}{(\arabic{enumi})}
\renewcommand{\labelenumi}{(\arabic{enumi})}
\item\label{cons:lyly} $l_y(u) = l_y(v)$ and $\vert[l_x(u), p_x(u)] \cap [l_x(v), p_x(v)]\vert > 1$, or
\item\label{cons:ryry} $r_y(u) = r_y(v)$ and $\vert[l_x(u), q_x(u)] \cap [l_x(v), q_x(v)]\vert > 1$, or
\item\label{cons:lyry} $l_y(u) = r_y(v)$ and $\vert[l_x(u), p_x(u)] \cap [l_x(v), q_x(v)]\vert > 1$, or
\item\label{cons:ryly} $l_y(v) = r_y(u)$ and $\vert[l_x(v), p_x(v)] \cap [l_x(u), q_x(u)]\vert > 1$, or
\item\label{cons:lxlx} $l_x(u) = l_x(v)$ and $\vert[l_y(u), r_y(u)] \cap [l_y(v), r_y(v)]\vert > 1$
\end{enumerate} 

We show that for distinct $u, v \in V(G)$, $uv \in E(G)$ if and only if $\mathcal{C}_u, \mathcal{C}_v$ satisfy any one of the conditions~\ref{cons:lyly} to~\ref{cons:lxlx}, thereby completing the proof.
Note that we only need to prove that $\mathcal{C}_u, \mathcal{C}_v$ satisfy any of the conditions~\ref{cons:lyly} to~\ref{cons:lxlx} if and only if $u$ is the parent of $v$, or $v$ is the parent of $u$, or $u, v$ are consecutive leaves. 

If $u$ is the parent of $v$, then from Observation~\ref{obs: comp}, $\mathcal{C}_u$ and $\mathcal{C}_v$ satisfy at least one of the conditions~\ref{cons:lyly}, \ref{cons:ryry} or \ref{cons:lyry}. When $u, v$ are consecutive leaves, from Observation~\ref{obs: leaves} we have that $\mathcal{C}_u, \mathcal{C}_v$ satisfy condition~\ref{cons:lyly} or~\ref{cons:ryry}. 

Conversely, we show that if $u$ is not a parent of $v$, $v$ is not a parent of $u$ and $u, v$ are not two consecutive leaves, then $\mathcal{C}_u,\mathcal{C}_v$ do not satisfy any one of the conditions~\ref{cons:lyly} to~\ref{cons:lxlx}. 

Suppose first that one of $u,v$ is an internal vertex. Without loss of generality assume that $u$ is an internal vertex. 

Suppose further that neither of $u,v$ is an ancestor of the other. Clearly, in this case neither $u$ nor $v$ is $r$. By Observation~\ref{obs: inc}(a), we know that $\mathcal{C}_u, \mathcal{C}_v$ do not satisfy condition~\ref{cons:lyly}. Observation~\ref{obs: inc}(a) also tells us that $l_x(u) \neq l_x(v)$ and therefore $\mathcal{C}_u, \mathcal{C}_v$ do not satisfy condition~\ref{cons:lxlx}. If $v \neq a_0$, then by Observation~\ref{obs: inc}(b), $\mathcal{C}_u, \mathcal{C}_v$ do not satisfy any of the conditions \ref{cons:ryry}, \ref{cons:lyry} and \ref{cons:ryly}. Let us suppose that $v = a_0$. Then $r_y(v)=h+1$. It is clear from our construction that since $u$ is an internal vertex other than $r$, $r_y(u) \neq h+1$ and therefore, $r_y(u) \neq r_y(v)$. This implies $\mathcal{C}_u, \mathcal{C}_v$ do not satisfy condition~\ref{cons:ryry}. Also, from our construction, there is no vertex $w$ such that $l_y(w)=h+1$, and therefore $\mathcal{C}_u, \mathcal{C}_v$ do not satisfy condition~\ref{cons:lyry}. Again, we have $l_y(v)=0$ and no vertex $w$ such that $r_y(w)=0$. Therefore $\mathcal{C}_u, \mathcal{C}_v$ do not satisfy condition~\ref{cons:ryly}.

Suppose that one of $u,v$ is an ancestor of the other. Without loss of generality assume that $u$ is an ancestor of $v$, but not the parent. Note that in this case, $u \neq a'$ and $v \neq a_0$. Clearly, $h^r(v) \geq h^r(u)+2$. Let $v \neq a_{k-1}$. From the construction it is clear that $r_y(v) = h-h^r(v)+1 < h-h^r(u) = l_y(u)$ and therefore, we have $[l_y(u), r_y(u)] \cap [l_y(v), r_y(v)] = \emptyset$. It follows that $\mathcal{C}_u, \mathcal{C}_v$ do not satisfy conditions~\ref{cons:lyly} to~\ref{cons:lxlx}. Let us consider the case when $v = a_{k-1}$. Suppose $u \neq r$. From the construction it is clear that $l_y(v) \neq l_y(u)$, $r_y(v) \neq r_y(u)$, $l_y(v) \neq r_y(u)$ and $r_y(v) \neq l_y(u)$. Also from the construction it is clear that $l_x(v) \neq l_x(u)$. Hence $\mathcal{C}_u, \mathcal{C}_v$ do not satisfy conditions~\ref{cons:lyly} to~\ref{cons:lxlx}. Suppose that $u = r$ and $v = a_{k-1}$. Then we have $l_x(u) \neq l_x(v)$, $l_y(u) \neq l_y(v)$, $l_y(u) \neq r_y(v)$ and $l_y(v) \neq r_y(u)$. Therefore, $\mathcal{C}_u$ and $\mathcal{C}_v$ do not satisfy conditions~\ref{cons:lyly}, \ref{cons:lyry}, \ref{cons:ryly} or \ref{cons:lxlx}. Also, since $q_x(u) = 1 < k-1 = l_x(v)$ (recall that that $r = u \neq a'$), we have $[l_x(u),q_x(u)] \cap [l_x(v),q_x(v)] = \emptyset$, implying that $\mathcal{C}_u$ and $\mathcal{C}_v$ do not satisfy condition~\ref{cons:ryry}.

Suppose $u$ and $v$ are two nonadjacent leaves. Then $l_x(u) \neq l_x(v)$. This implies $\mathcal{C}_u, \mathcal{C}_v$ do not satisfy condition~\ref{cons:lxlx}. Again $l_y(u) = 0$ and $l_y(v) = 0$. Hence, $l_y(u) \neq r_y(v)$ and $l_y(v) \neq r_y(u)$. This implies $\mathcal{C}_u, \mathcal{C}_v$ do not satisfy conditions~\ref{cons:lyry} and~\ref{cons:ryly}. As $u, v$ are nonadjacent leaves, from the construction we have $[l_x(u), p_x(u)] \cap [l_x(v), p_x(v)] \neq \emptyset$. This implies $\mathcal{C}_u, \mathcal{C}_v$ do not satisfy condition~\ref{cons:lyly}. If $u, v \notin \{a_0, a_{k-1}\}$, from the construction we have $[l_x(u), q_x(u)] \cap [l_x(v), q_x(v)] \neq \emptyset$. If either of $u$ or $v$ belongs to $\{a_0, a_{k-1}\}$ then $r_y(u) \neq r_y(v)$. This implies $\mathcal{C}_u, \mathcal{C}_v$ do not satisfy condition~\ref{cons:ryry}.

This completes the proof of Claim~\ref{clm: halin2epg}.
\end{proof}

\begin{figure}[h]
\centering
\scalebox{0.6}
{
\begin{pspicture}(0,-3.0775)(5.61,3.0475)
\psdots[dotsize=0.12](2.8,-1.9625)
\psline[linewidth=0.02cm,dotsize=0.07055555cm 2.0]{*-}(2.8,0.0375)(2.8,-0.9625)
\psline[linewidth=0.02cm,dotsize=0.07055555cm 2.0]{*-}(2.8,-0.9625)(2.8,-1.9625)
\psline[linewidth=0.02cm,dotsize=0.07055555cm 2.0]{*-}(3.3,-2.6625)(2.8,-1.9625)
\psline[linewidth=0.02cm,dotsize=0.07055555cm 2.0]{*-}(1.8,-0.1625)(2.8,0.0375)
\psline[linewidth=0.02cm,dotsize=0.07055555cm 2.0]{*-}(3.8,-0.1625)(2.8,0.0375)
\psline[linewidth=0.02cm,dotsize=0.07055555cm 2.0]{*-}(3.4,1.0375)(2.8,0.0375)
\psline[linewidth=0.02cm,dotsize=0.07055555cm 2.0]{*-}(2.2,1.0375)(2.8,0.0375)
\psline[linewidth=0.02cm,dotsize=0.07055555cm 2.0]{*-}(2.3,-2.6625)(2.8,-1.9625)
\psline[linewidth=0.02cm,dotsize=0.07055555cm 2.0]{*-}(5.6,0.0375)(4.801983,-0.35836658)
\psline[linewidth=0.02cm,dotsize=0.07055555cm 2.0]{*-}(5.3,-1.0625)(4.801983,-0.35836658)
\psline[linewidth=0.02cm,dotsize=0.07055555cm 2.0]{*-}(4.8,-0.3625)(3.8,-0.1625)
\psline[linewidth=0.02cm,dotsize=0.07055555cm 2.0]{*-}(3.9,1.9375)(3.4,1.0375)
\psline[linewidth=0.02cm,dotsize=0.07055555cm 2.0]{*-}(1.7,1.9375)(2.2,1.0375)
\psline[linewidth=0.02cm,dotsize=0.07055555cm 2.0]{*-}(0.8,-0.3625)(1.8,-0.1625)
\psline[linewidth=0.02cm,dotsize=0.07055555cm 2.0]{*-}(4.7,2.3375)(3.9,1.9375)
\psline[linewidth=0.02cm,dotsize=0.07055555cm 2.0]{*-}(0.0,0.0375)(0.7928409,-0.38716066)
\psline[linewidth=0.02cm,dotsize=0.07055555cm 2.0]{*-}(0.8,-1.6625)(1.8,-0.1625)
\psline[linewidth=0.02cm,dotsize=0.07055555cm 2.0]{*-}(1.6,-2.3625)(2.8,-0.9625)
\psline[linewidth=0.02cm,dotsize=0.07055555cm 2.0]{*-}(4.0,-2.3625)(2.8,-0.9625)
\psline[linewidth=0.02cm,dotsize=0.07055555cm 2.0]{*-}(4.8,-1.6625)(3.8,-0.1625)
\psline[linewidth=0.02cm,dotsize=0.07055555cm 2.0]{*-}(5.5,0.7375)(3.8,-0.1625)
\psline[linewidth=0.02cm,dotsize=0.07055555cm 2.0]{*-}(5.1,1.7375)(3.4,1.0375)
\psline[linewidth=0.02cm,dotsize=0.07055555cm 2.0]{*-}(3.2,3.0375)(3.4,1.0375)
\psline[linewidth=0.02cm,dotsize=0.07055555cm 2.0]{*-}(2.4,3.0375)(2.2,1.0375)
\psline[linewidth=0.02cm,dotsize=0.07055555cm 2.0]{*-}(0.5,1.7375)(2.2,1.0375)
\psline[linewidth=0.02cm,dotsize=0.07055555cm 2.0]{*-}(0.1,0.7375)(1.8,-0.1625)
\psline[linewidth=0.02cm,dotsize=0.07055555cm 2.0]{*-}(1.7,2.9375)(1.7,1.9375)
\psline[linewidth=0.02cm,dotsize=0.07055555cm 2.0]{*-}(0.9,2.3375)(1.7,1.9375)
\psline[linewidth=0.02cm,dotsize=0.07055555cm 2.0]{*-}(3.9,2.9375)(3.9,1.9375)
\psline[linewidth=0.02cm,dotsize=0.07055555cm 2.0]{*-}(0.3,-1.0625)(0.8,-0.3625)
\psline[linewidth=0.02cm](3.9,2.9375)(4.7,2.3375)
\psline[linewidth=0.02cm](1.7,2.9375)(2.4,3.0375)
\psline[linewidth=0.02cm](0.9,2.3375)(1.7,2.9375)
\psline[linewidth=0.02cm](4.0,-2.3625)(4.8,-1.6625)
\psline[linewidth=0.02cm](4.8,-1.6625)(5.3,-1.0625)
\psline[linewidth=0.02cm](2.4,3.0375)(3.2,3.0375)
\psline[linewidth=0.02cm](3.2,3.0375)(3.9,2.9375)
\psline[linewidth=0.02cm](5.5,0.7375)(5.6,0.0375)
\psline[linewidth=0.02cm](5.6,0.0375)(5.3,-1.0625)
\psline[linewidth=0.02cm](0.8,-1.6625)(1.6,-2.3625)
\psline[linewidth=0.02cm](0.3,-1.0625)(0.8,-1.6625)
\psline[linewidth=0.02cm](0.1,0.7375)(0.0,0.0375)
\psline[linewidth=0.02cm](0.0,0.0375)(0.3,-1.0625)
\psline[linewidth=0.02cm](3.3,-2.6625)(4.0,-2.3625)
\psline[linewidth=0.02cm](1.6,-2.3625)(2.3,-2.6625)
\psline[linewidth=0.02cm](2.3,-2.6625)(3.3,-2.6625)
\psline[linewidth=0.02cm](4.7,2.3375)(5.1,1.7375)
\psline[linewidth=0.02cm](5.1,1.7375)(5.5,0.7375)
\psline[linewidth=0.02cm](0.9,2.3375)(0.5,1.7375)
\psline[linewidth=0.02cm](0.5,1.7375)(0.1,0.7375)
\rput(3.0545313,-0.2475){$v_0$}
\rput(3.05045313,-0.8475){$v_1$}
\rput(3.6045313,0.0525){$v_2$}
\rput(3.6045313,0.8525){$v_3$}
\rput(2.5045312,0.9525){$v_4$}
\rput(1.9545311,0.0525){$v_5$}
\rput(3.0545313,-1.8475){$v_7$}
\rput(1.6045312,-2.5475){$v_8$}
\rput(2.5045313,-2.8475){$v_9$}
\rput(3.5345312,-2.8475){$v_{10}$}
\rput(4.234531,-2.5475){$v_{11}$}
\end{pspicture} 
}
\caption{A Halin graph which has no $B_1$-EPG representation.}
\label{fig: fig02}
\end{figure}
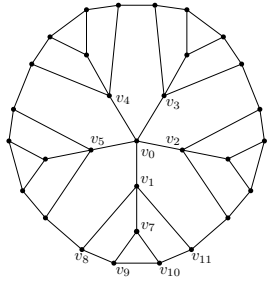

Now, we prove that there are Halin graphs that do not have $B_1$-EPG representations.

In a $B_1$-EPG representation, a one bend curve looks either like an L-shaped curve, or its rotations. Let us denote a one bend curve as $\mathcal{L}'(x_1, x_2, y_1, y_2, (x_i, y_j)) = \{(x,y)~|~y = y_j~\mbox{and}~x \in [x_1, x_2]\}\cup\{(x,y)~|~x=x_i~\mbox{and}~y\in [y_1,y_2]\}$, where $i, j\in \{1,2\}$. Clearly an L-shaped curve is $\mathcal{L}'(x_1, x_2, y_1, y_2, (x_1, y_1))$. Suppose that there is a $B_1$-EPG representation for a graph $H$ in which each vertex $u\in V(H)$ is represented by a one-bend curve $L'_u$. With respect to this representation, for any vertex $u\in V(H)$ define $\mathcal{L}'_u = \mathcal{L}'(x_1, x_2, y_1, y_2, (x_i, y_j))$, where $i, j \in \{1,2\}$. Also define $l_x(u)=x_1, r_x(u) = x_2, l_y(u)=y_1$ and $r_y(u)=y_2$. Let us denote the bend point $(x_i, y_j)$ by $\mathtt{b}_u$. Let us also denote the $x$-coordinate of $\mathtt{b}_u$ by $x(\mathtt{b}_u)$ and $y$-coordinate of $\mathtt{b}_u$ by $y(\mathtt{b}_u)$. It is easy to see that $\mathcal{L}'_u \cap \mathcal{L}'_v$ contains a horizontal or a vertical line segment of non-zero length if and only if $\mathcal{L}'_u, \mathcal{L}'_v$ satisfy at least one of the following conditions. 

\begin{enumerate}
\renewcommand{\theenumi}{(\arabic{enumi})}
\renewcommand{\labelenumi}{(\arabic{enumi})}
\item\label{cons:xx} $x(\mathtt{b}_u) = x(\mathtt{b}_v)$ and $\vert[l_y(u), r_y(u)] \cap [l_y(v), r_y(v)]\vert > 1$, or
\item\label{cons:yy} $y(\mathtt{b}_u) = y(\mathtt{b}_v)$ and $\vert[l_x(u), r_x(u)] \cap [l_x(v), r_x(v)]\vert > 1$
\end{enumerate} 

Since the curves $\{\mathcal{L}'_u\}_{u\in V(H)}$ form a valid $B_1$-EPG representation of the graph $H$, for every $u, v \in V(H)$, $uv \in E(H)$ if and only if $\mathcal{L}'_u, \mathcal{L}'_v$ satisfy at least one of the conditions~\ref{cons:xx} or \ref{cons:yy}.

First we prove the following observation for later use.

\begin{obs}
\label{obs:cont}
Suppose that the curves $\{\mathcal{L}'_{u_i}\}_{i\in \{1,2,3,4\}}$ form a $B_1$-EPG representation of the chordless cycle on four vertices denoted by $u_1u_2u_3u_4u_1$. If $\mathcal{L}'_{u_1}, \mathcal{L}'_{u_2}$ and $\mathcal{L}'_{u_1},\mathcal{L}'_{u_4}$ both satisfy condition~\ref{cons:xx} or both satisfy condition~\ref{cons:yy}, then $\mathtt{b}_{u_2}= \mathtt{b}_{u_4}$.
\end{obs}

\begin{proof}
Let us first consider the case when $\mathcal{L}'_{u_1}, \mathcal{L}'_{u_2}$ and $\mathcal{L}'_{u_1},\mathcal{L}'_{u_4}$ both satisfy condition~\ref{cons:xx}. In this case, we have $x(\mathtt{b}_{u_1})=x(\mathtt{b}_{u_2})=x(\mathtt{b}_{u_4})$. As $u_2$ and $u_4$ are nonadjacent, $\mathcal{L}'_{u_2}$ and $\mathcal{L}'_{u_4}$ do not satisfy condition~\ref{cons:xx}, which implies that $|[l_y(u_2),r_y(u_2)]\cap [l_y(u_4),r_y(u_4]|\leq 1$. Therefore, we have either $r_y(u_2)\leq l_y(u_4)$ or $r_y(u_4)\leq l_y(u_2)$. We can assume without loss of generality that $r_y(u_2)\leq l_y(u_4)$ (as the cycle can be relabeled to make this true if it is not already the case). Clearly, this implies $l_y(u_1)<r_y(u_2)\leq l_y(u_4)<r_y(u_1)$. Now suppose that $x(\mathtt{b}_{u_3}) = x(\mathtt{b}_{u_1})$ ( $= x(\mathtt{b}_{u_2}) = x(\mathtt{b}_{u_4})$). As $u_1$ and $u_3$ are nonadjacent, $\mathcal{L}'_{u_1}$ and $\mathcal{L}'_{u_3}$ should not satisfy condition~\ref{cons:xx}, implying that either $r_y(u_3)\leq l_y(u_1)$ or $r_y(u_1)\leq l_y(u_3)$. If $r_y(u_3)\leq l_y(u_1)$, we have by our previous observation that $r_y(u_3) < l_y(u_4)$. But this would imply that $\mathcal{L}'_{u_3}, \mathcal{L}'_{u_4}$ do not satisfy either condition~\ref{cons:xx} or condition~\ref{cons:yy}, which contradicts the fact that $u_3$ and $u_4$ are adjacent. Similarly, if $r_y(u_1)\leq l_y(u_3)$, we have $r_y(u_2) < l_y(u_3)$, which means that $\mathcal{L}'_{u_2}, \mathcal{L}'_{u_3}$ will not satisfy either condition~\ref{cons:xx} or condition~\ref{cons:yy}, which is a contradiction as $u_2$ and $u_3$ are adjacent. Therefore, we can conclude that $x(\mathtt{b}_{u_3}) \neq x(\mathtt{b}_{u_1})$. Since this means that $x(\mathtt{b}_{u_3}) \neq x(\mathtt{b}_{u_2})$ and $x(\mathtt{b}_{u_3})\neq x(\mathtt{b}_{u_4})$, it must be the case that both $\mathcal{L}'_{u_3}, \mathcal{L}'_{u_2}$ and $\mathcal{L}'_{u_3}, \mathcal{L}'_{u_4}$ satisfy condition~\ref{cons:yy}. Then $y(\mathtt{b}_{u_3}) = y(\mathtt{b}_{u_2})$ and $y(\mathtt{b}_{u_3}) = y(\mathtt{b}_{u_4})$ together implies $y(\mathtt{b}_{u_2}) = y(\mathtt{b}_{u_4})$. As $x(\mathtt{b}_{u_2}) = x(\mathtt{b}_{u_4})$ we conclude that $\mathtt{b}_{u_2} = \mathtt{b}_{u_4}$.
 
We can argue similarly for the case when $\mathcal{L}'_{u_1}, \mathcal{L}'_{u_2}$ and $\mathcal{L}'_{u_1}, \mathcal{L}'_{u_4}$ both satisfy condition~\ref{cons:yy}.
\end{proof}

We consider $H$ to be the Halin graph with thirty one vertices and fifty edges, shown in the Figure~\ref{fig: fig02}. We show that this graph has no $B_1$-EPG representation, thereby completing the proof the Claim~\ref{clm: halinb1}.

\begin{clm}
\label{clm: halinb1}
There exists a Halin graph that cannot be represented as an edge intersection graph of paths on a grid such that each path has at most one bend.
\end{clm}
\begin{proof}
For the sake of contradiction assume that $H$ has a $B_1$-EPG representation. This implies, with every vertex $v\in V(G)$, we can associate a one bend curve $\mathcal{L}'_v$, such that $uv \in E(G)$ if and only if $\mathcal{L}'_u, \mathcal{L}'_v$ satisfy at least one of the conditions~\ref{cons:xx} or \ref{cons:yy}. Note that the vertex $v_0$ has five neighbours $v_1, v_2, \dots , v_5$. Clearly, each pair of curves $\mathcal{L}'_{v_0}, \mathcal{L}'_{v_i}$ satisfies at least one of the conditions~\ref{cons:xx} or \ref{cons:yy} for each value of $i\in\{1,2,\dots,5\}$. Note that we can assume that a majority of these pairs of curves satisfy condition~\ref{cons:yy} (as if that is not the case we can reflect the whole collection of curves along the line $x=y$ to obtain another valid representation of the graph in which this is true). So we assume without loss of generality that each pair $\mathcal{L}'_{v_0}, \mathcal{L}'_{v_i}$, for $1\leq i\leq 3$ satisfies condition~\ref{cons:yy}. That is, $y(\mathtt{b}_{v_0})=y(\mathtt{b}_{v_i})$ and $|[l_x(v_0),r_x(v_0)]\cap [l_x(v_i),r_x(v_i)]| > 1$, for each $i\in\{1,2,3\}$. As $\{v_1,v_2,v_3\}$ form an independent set in the graph, there exists a vertex, say $v_1$, such that $[l_x(v_1), r_x(v_1)] \subsetneq [l_x(v_0), r_x(v_0)]$. This means that if any pair of curves $\mathcal{L}'_{v_1}, \mathcal{L}'_{v_i}$, where $i\in\{7,8,11\}$, satisfy condition~\ref{cons:yy}, then $\mathcal{L}'_{v_0}, \mathcal{L}'_{v_i}$ also satisfy condition~\ref{cons:yy}, which is a contradiction as none of the vertices $v_7,v_8$ or $v_{11}$ are adjacent to $v_0$. Therefore, for each $i\in\{7,8,11\}$, the curves $\mathcal{L}'_{v_1}, \mathcal{L}'_{v_i}$ must satisfy the condition~\ref{cons:xx}. Consider the cycle $v_1v_8v_9v_7v_1$. Then from the Observation~\ref{obs:cont} we have $\mathtt{b}_{v_7} = \mathtt{b}_{v_8}$. Also $v_1v_7v_{10}v_{11}v_1$ form a cycle. Again applying Observation~\ref{obs:cont} we have $\mathtt{b}_{v_7} = \mathtt{b}_{v_{11}}$. This implies that $|[l_y(v_i),r_y(v_i)] \cap [l_y(v_j),r_y(v_j)]| > 1$ for some distinct $i,j \in \{7,8,11\}$. As we have $x(\mathtt{b}_{v_1})=x(\mathtt{b}_{v_7})=x(\mathtt{b}_{v_8})=x(\mathtt{b}_{v_{11}})$, this means that $\mathcal{L}'_{v_i}, \mathcal{L}'_{v_j}$ satisfy condition~\ref{cons:xx}, which is a contradiction to the fact that $\{v_7,v_8,v_{11}\}$ form an independent set in the graph.

Hence the proof.
\end{proof}

\begin{thm}
Tree-union-cycle graphs are in $B_2$-EPG. There are Halin graphs that are not in $B_1$-EPG.
\end{thm}
\begin{proof}
The proof follows directly from Claims~\ref{clm: halin2epg} and~\ref{clm: halinb1}.
\end{proof}

\begin{cor}
If $H$ is any Halin graph, then $b_e(H)\leq 2$. This bound is tight.
\end{cor}
\section{Concluding remarks}
We showed that the tree-union-cycle graph $G$ has a $B_2$-EPG representation using C-shaped curves. It can be seen that $G$ has a $B_2$-EPG representation using S-shaped curves too.

An S-shaped curve is defined as a set of points given by $\mathcal{S}(x_1, x_2, x_3, y_1, y_2) = \{(x, y)~\vert~x = x_2\ \mbox{and}\ y \in [y_1, y_2]\} \cup \{ (x, y)~\vert~ y = y_1\ \mbox{and}\ x \in [x_1, x_2]\} \cup \{(x, y)~\vert~ y = y_2\ \mbox{and}\ x \in [x_2, x_3]\}$. Our aim is to associate an S-shaped curve $\mathcal{S}_u$ to each vertex $u\in V(G)$ such that $uv\in E(G)$ if and only if $\mathcal{S}_u\cap \mathcal{S}_v$ contains at least a horizontal or a vertical line segment of non-zero length. For any vertex $u \in V(G)$ with $\mathcal{S}_u = \mathcal{S}(x_1,x_2,x_3,y_1,y_2)$, define $l_x(u)=x_1, m_x(u)=x_2, r_x(u) = x_3, l_y(u)=y_1$ and $r_y(u)=y_2$.

We choose the root $r$ and label the leaves $a_0,a_1,\ldots,a_{k-1}$ as described in the previous section. Define $h = \max \{ h^{r}(u)~|~u \in V(G)\}$. Let $\epsilon = 1/2h$.
\medskip

For every leaf $a_i$ other than $a_0$ or $a_{k-1}$, define $\mathcal{S}_{a_i} = \mathcal{S}(i-1-\epsilon, i, i+1 - \epsilon, 0, h - h^{r}(a_i)+1)$.

Define $\mathcal{S}_{a_0} = \mathcal{S}(-\epsilon, 0, k-1 + 3\epsilon, 0, h+1)$.

Define $\mathcal{S}_{a_{k-1}} = \mathcal{S}(k-2-\epsilon, k-1, k-1+\epsilon, 0, h+1)$.

For every internal vertex $v$ other than $r$ and $a'$, define $\mathcal{S}_v = \mathcal{S}(\min\{i~\vert~ a_i \in L^{r}(v)\}, \max\{i+(h-h^{r}(v))\epsilon \vert~ a_i \in L^{r}(v)\},\max\{i+1-\epsilon~\vert~ a_i \in L^{r}(v)\}, h-h^{r}(v), h-h^{r}(v)+1)$. 

If $a'\neq r$, define $\mathcal{S}_{a'} = \mathcal{S}(k-1-\epsilon, k-1+\epsilon, k-1+2\epsilon, 0, h - h^{r}(a')+1)$ and define $\mathcal{S}_r=\mathcal{S}(\epsilon,k-1+2\epsilon,k-1+3\epsilon,h,h+1)$.

If $a'=r$, define $\mathcal{S}_r=\mathcal{S}(\epsilon,k-1-\epsilon,k-1+\epsilon,h,h+1)$.

It is easy to see that $\mathcal{S}_u \cap \mathcal{S}_v$ contains a horizontal or a vertical line segment of non-zero length if and only if $\mathcal{S}_u, \mathcal{S}_v$ satisfies at least one of the following.
\begin{enumerate}
\renewcommand{\theenumi}{(\arabic{enumi})}
\renewcommand{\labelenumi}{(\arabic{enumi})}
\item\label{cons:slyly} $l_y(u) = l_y(v)$ and $\vert[l_x(u), m_x(u)] \cap [l_x(v), m_x(v)]\vert > 1$, or
\item\label{cons:sryry} $r_y(u) = r_y(v)$ and $\vert[m_x(u), r_x(u)] \cap [m_x(v), r_x(v)]\vert > 1$, or
\item\label{cons:slyry} $l_y(u) = r_y(v)$ and $\vert[l_x(u), m_x(u)] \cap [m_x(v), r_x(v)]\vert > 1$, or
\item\label{cons:sryly} $l_y(v) = r_y(u)$ and $\vert[l_x(v), m_x(v)] \cap [m_x(u), r_x(u)]\vert > 1$, or
\item\label{cons:smxmx} $m_x(u) = m_x(v)$ and $\vert[l_y(u), r_y(u)] \cap [l_y(v), r_y(v)]\vert > 1$
\end{enumerate} 

As we did in Section \ref{sec: b2}, it can be verified that for distinct $u, v \in V(G)$, $uv \in E(G)$ if and only if $\mathcal{S}_u, \mathcal{S}_v$ satisfy at least one of the conditions~\ref{cons:slyly} to~\ref{cons:smxmx}. Thus the collection of S-shaped curves $\{\mathcal{S}_v\}_{v\in V(G)}$ forms a $B_2$-EPG representation of the graph $G$.

\bibliographystyle{alpha}
\newcommand{\etalchar}[1]{$^{#1}$}

\end{document}